\documentclass{amsart}

\usepackage{graphicx}
\usepackage{dsfont}
\usepackage{tikz}
\usepackage{srcltx}
\usepackage{hyperref}

\newtheorem{theorem}{Theorem}[section]
\newtheorem{proposition}[theorem]{Proposition}

\theoremstyle{definition}

\theoremstyle{remark}
\newtheorem{remark}[theorem]{Remark}

\numberwithin{equation}{section}

\newcommand{\Z}{\mathbb{Z}}
\newcommand{\N}{\mathbb{N}}
\newcommand{\abs}[1]{\lvert#1\rvert}

\begin{document}

\title{Well-balanced and asymptotic preserving schemes for kinetic models}

\author{Casimir EMAKO}
\address{Sorbonne Universit\'es, UPMC Univ Paris 06, UMR 7598, Laboratoire Jacques-Louis Lions, F-75005, Paris, France\\
CNRS, UMR 7598, Laboratoire Jacques-Louis Lions, F-75005, Paris, France \\
INRIA-Paris-Rocquencourt, EPC MAMBA, Domaine de Voluceau, BP105, 78153 Le Chesnay Cedex}
\email{emako@ann.jussieu.fr}

\author{Min Tang}
\address{Institute of natural sciences and department of mathematics, Shanghai Jiao Tong University, Shanghai, 200240,China}
\email{tangmin@sjtu.edu.cn}


\date{\today}


\begin{abstract}
In this paper, we propose a general framework for designing numerical schemes that have both well-balanced (WB) and asymptotic preserving (AP) properties,
for various kinds of kinetic models. We are interested in two different parameter regimes, 1)  When the ratio between the mean free path and the characteristic macroscopic length $\varepsilon$ tends to zero, the density can be described by (advection) diffusion type (linear or nonlinear) macroscopic models; 2)
When $\varepsilon=O(1)$, the models behave like hyperbolic equations with source terms and we are interested in their steady states. 
We apply the framework to three different kinetic models: neutron transport equation and its diffusion limit, the transport equation for chemotaxis and its Keller-Segel limit,
and grey radiative transfer equation and its nonlinear 
diffusion limit. Numerical examples are given to demonstrate the properties of the schemes. 

\end{abstract}

\maketitle

\section{Introduction}
Transport equations are important since they arise in many important applications, ranging from neutron transport, radiative transfer, semiconductor device simulation to E.coli chemotaxis. These models describe particles that travel freely for a certain distance and then change their directions by 
scattering, interacting with background media or tumbling. The mean free path (the average distance a particle travels between two successive velocity changes) is an important
parameter, when it is small compared to the typical length scales, various macroscopic model can be derived asymptotically or analytically \cite{CMPS,LMM,OthmerHill}.

Let the dimensionless parameter $\varepsilon$ denote the ratio of the mean free path and the typical length scale.
Numerical solutions to the transport equation are challenging when $\varepsilon$ is small, since it requires the numerical resolution of the small scale.
To develop a multi-scale scheme whose stability and convergence are independent of $\varepsilon$ refers to the asymptotic preserving (AP) property.
When $\varepsilon=O(1)$, the models behave like hyperbolic equations with source terms.
When the
source terms in the system become stiff, the usual numerical methods may give
poor approximations to the steady state solutions \cite{BPV,Leveque}. To maintain the steady states or to achieve them in the long time limit with an acceptable level of accuracy refers to the well balanced (WB) property.

In this work, we present a general framework to build schemes that have both properties. 
Two different kinetic models are considered to illustrate the idea of our proposed framework

\bigskip
\textbf{Transport equation for Chemotaxis and its Keller-Segel limit}\\
Bacteria undergo run and tumble process as mentioned in \cite{Berg,VCJS, Saragosti}. During the run phase, bacteria move along a straight line and change their 
directions during the tumble phase. This individual motion is called the velocity jump process and modeled by the so-called Othmer-Dunbar-Alt model \cite{Alt,OthmerAlt} which reads 
\begin{equation}\label{eq_chemotaxis_eps}
\left\{
 \begin{aligned}
  & \partial_t f^{\varepsilon}+\frac{v}{\varepsilon}\cdot \nabla_x f^{\varepsilon}=\frac{1}{\varepsilon^2}\left(\frac{1}{|V|}\int_{V} \big(1+\varepsilon \phi(v'\cdot\nabla_x S)\big)f^{\varepsilon}(v')dv'- \big(1+\varepsilon \phi(v\cdot\nabla_x S)\big)f^{\varepsilon}(v)\right),\\
  & \partial_t S^{\varepsilon}-D\Delta S^{\varepsilon}+\alpha S^{\varepsilon}=\beta \rho^{\varepsilon}(x,t):=\frac{1}{|V|}\int_{V} f^{\varepsilon}dv.
 \end{aligned}
 \right.
\end{equation}
Here $f^{\varepsilon}(x,v,t)$ is the amount of cells with velocity $v\in V$ ($V$ is a sphere) at position $x$ and time $t>0$, $\phi(u)$ is a decreasing function in $u$,
$S^{\varepsilon}(x,t)$ is the concentration of the chemical substance.
The parameters $D,\alpha,\beta$ are positive constants and $\varepsilon$ is the ratio of the average run distance between two successive tumbles and the characteristic macroscopic length.

When $\epsilon\to 0$, the solution of \eqref{eq_chemotaxis_eps} $f^{\varepsilon}(x,v,t)$ tends to $\rho^0(x,t)$,
which is independent of $v$ and satisfies the following Keller-Segel type equation \cite{CMPS,Almeida,HKS,Keller,OthmerHill}~:
\begin{equation}\label{limit_eq}
\left\{
\begin{aligned}
 & \partial_t \rho=\frac{1}{3}\Delta\rho +\nabla\left(\Big(\frac{1}{|V|}\int_{V} v\phi(v\partial_x S)dv \Big)\rho \right),\\
 & \partial_t S-D \Delta S+\alpha S=\beta \rho.
\end{aligned}
\right.
\end{equation}
Many numerical schemes have been proposed to study \eqref{eq_chemotaxis_eps} \cite{Bokai,FilbetY,Gosse1}. 

\bigskip
\textbf{Gray radiative transfer equation and its nonlinear diffusion limit:}\\
The gray radiative transfer equation concerns photon transport and its interaction with the 
background material. It has wide application in astrophysics and inertial confinement fusion.
The system for the radiative intensity $I$ and the material temperature is
\begin{equation}
\left\{
\begin{aligned}
& \frac{1}{c} \partial_t I+\frac{1}{\varepsilon} v\cdot \nabla_x I=
\frac{\sigma}{\varepsilon^2}\left(\frac{1}{|V|}acT^4-I\right)+q(v),\\
&  C_v\partial_t T=\frac{\sigma}{\varepsilon^2}(\rho-acT^4),
\end{aligned}
\right.
\label{radiative_transport}
\end{equation}
where $\rho:=\int_{V}I dv$, $\sigma(x,T)$ is the opacity and $a$, $c$, $C_v$ are positive constants
represent the radiation constant, light speed and heat capacity respectively.
Similar to the chemotaxis case, when $\varepsilon$ goes to zero, the radiative intensity $I$ approaches
a Planckian at the local temperature $\frac{1}{\abs{V}}acT^4$ and the  
photon temperature satisfies the following nonlinear diffusion equation: 
\begin{equation}\label{eq:limit_T}
a \partial_t T^4+C_v \partial_t T=\nabla_x (\frac{ac}{3\sigma}\nabla_x T^4)+\int_Vq\, dv.
\end{equation}
It is important to preserve the steady state temperature distribution when $t\to\infty$. 
\bigskip

 A scheme for such problems is AP if it preserves the limiting equation \eqref{limit_eq} or \eqref{eq:limit_T} when $\varepsilon\to 0$ at the discrete level.
 AP schemes were first designed for the neutron transport equation to use unresolved cells to capture the macroscopic diffusion limit model \cite{LM,LMM}. It has been successfully extended to a lot of applications, we refer to the review paper \cite{SJin} for more discussions. 
In the literature, many papers have been devoted to build AP schemes for diffusion limit of the transport equations, for example \cite{Adams,Bokai,Ribot,Mieussens}.

WB schemes are developed for hyperbolic equations with source terms. They 
have been proposed for various applications, including the Saint-Venant system with a source term, 
the non-isothermal nozzle flow equations, etc.
 By balancing the numerical flux with the
source term, WB schemes can capture the steady state solutions exactly
or with at least a second order accuracy \cite{JW}.
A very limited list of such references includes but not limits to quasi-steady scheme \cite{Leveque}, kinetic schemes \cite{PS,Xu1}, central schemes \cite{Kurganov1}, interface scheme \cite{JW}, and other references there in
\cite{BAMN}.

Recently designing schemes that are either AP or WB, or have both properties attracts a lot of interests, both for kinetic chemotaxi model \cite{Gosse1,Gosse, Natalini} and grey radiative
transfer equation \cite{Gosse2,Sun}. 
Generally, it is very hard to have a scheme that have both properties and this is the goal of our present paper.

The paper is organized as follows: In section 2, the scheme framework composed of two steps: prediction step and steady problem step, is described and we show that the
AP and WB properties can be achieved.
 Section 3 and 4 are respectively devoted to the construction of 
AP and WB schemes for the transport equation for chemotaxis and the gray radiative transfer equation. In the prediction step, we first extend the unified gas approach for neutron transport equation \cite{Mieussens} to
the chemotaxis model \eqref{eq_chemotaxis_eps} and construct a new AP scheme for E.coli chemotaxis, while, for the gray radiative transfer equation, the scheme in 
\cite{Sun} is employed. Then, in the steady problem step, we use the numerical results
obtained by an AP scheme for the steady state equation of \eqref{eq_chemotaxis_eps} or \eqref{radiative_transport} to modify the numerical flux.  
The performances of the proposed schemes are presented in section 5 and we conclude with some discussion in section 6.

\section{The scheme framework and its WB and AP properties}
In this part, we introduce a general framework of designing WB and AP schemes, while the details of the discretization are given in section 3 and section 4. 
In the subsequent part of the paper, we consider the one dimensional problem and use a uniform grid with
 $$x_i=i\Delta x,\quad i\in \Z,\qquad t^n=n\Delta t,\quad n\in \N.$$
Extensions to the two dimensional case are straightforward.

To illustrate the idea, we consider the following simplest one dimensional neutron transport equation
\begin{equation}\label{eq:neutrontrans}
 \partial_t f+\frac{1}{\varepsilon} v \partial_x f=
\frac{\sigma_T}{\varepsilon^2}\left(\rho-f\right)-\sigma_a\rho+q,
\end{equation}with $\rho:=\frac{1}{2}\int_{-1}^1f dv$.
When $\varepsilon\to 0$, the solution of the above equation tends to $\rho_0$ which satisfies the following
diffusion equation
$$
\partial_t\rho_0-\partial_x\big(\frac{1}{3\sigma_T}\partial_x\rho_0\big)+\sigma_a\rho_0=q.
$$
The formal derivation is standard by substituting the Chapman-Enskog expansion
$$
f=f^{(0)}+\varepsilon f^{(1)}+\varepsilon^2f^{(2)}+\cdots
$$into \eqref{eq:neutrontrans}. When $\varepsilon$ is small, the solution to \eqref{eq:neutrontrans}
$f$ can be approximated by
\begin{equation}
f=\rho^{(0)}-\frac{\varepsilon}{\sigma_T}v\partial_x\rho^{(0)}+O(\varepsilon^2).
\end{equation}

 We are interested in $\varepsilon$ ranging from $O(1)$ to very small,
the numerical scheme writes:
\begin{equation}\label{numerical_scheme}
\begin{aligned}
 & \frac{f_i^{n+1}-f_i^n}{\Delta t}+\frac{v}{\varepsilon \Delta x}\left((1-\tilde{\alpha})f_i^n+\tilde{\alpha} f^{n+1}_i-\hat{f}_{i-1/2}^{n}\right)=0,\quad v>0,\\
 & \frac{f_i^{n+1}-f_i^n}{\Delta t}+\frac{v}{\varepsilon \Delta x}\left(\hat{f}^n_{i+1/2}-\left((1-\tilde{\alpha})f_i^n+\tilde{\alpha} f_i^{n+1}\right)\right)=0,\quad v<0,
\end{aligned}
\end{equation}
where $f_i^n$ are the approximation of the average $f(x,t^n)$ in the interval $(x_{i-1/2},x_{i+1/2})=(x_i-\Delta x/2,x_i+\Delta x/2)$ and 
\begin{equation}\tilde{\alpha}=\min\left(1,\frac{\Delta t}{\varepsilon}\right).\label{eq:tildealpha}\end{equation}  The determination of $\hat{f}^n_{i+1/2}$ consists of two steps: the prediction step and the steady problem step.

\bigskip
\textbf{Prediction step: }
The prediction step is to evolve the equation \eqref{eq:neutrontrans} for one time step by an AP scheme.
Starting from $f_i^n$ obtained from the $n$th iteration, the predictions $\tilde{f}_i^{n+1}$ and 
$\tilde{\rho}_i^{n+1}=\frac{1}{2}\int_{-1}^1f_i^{n+1}dv$ can be found by any scheme that is AP.  The requirement that,
when $\varepsilon\to 0$, the scheme preserves the diffusion limit at the discrete level can only be achieved when 
\begin{equation}\label{eq:tildefrho}\tilde{f}_i^{n+1}\approx \tilde{\rho}_i^{n+1}+O(\varepsilon)\end{equation} and $\tilde{\rho}_i^n$ is a discretization of the limiting diffusion equation. Here in this paper we use the unified gas approach developed in \cite{Mieussens,Xu,Sun} for various kinetic models.

\bigskip
\textbf{Steady problem step: }
The second step is devoted to the computation of the steady state problem. 
 On each cell $[x_i,x_{i+1}]$, we solve the following stationary problem:
 \begin{equation} v\cdot \partial_x \hat{f}=
\frac{\sigma_T}{\varepsilon}\left(\hat{\rho}-\hat{f}\right)-\varepsilon\sigma_a\hat{\rho}+\varepsilon q
\label{eq:hatf}
\end{equation}
together with inflow boundary conditions 
\begin{equation}\label{eq:stableinflow}
\begin{cases}
&\hat{f}(x_i,v)=(1-\tilde{\alpha})f_i^n+\tilde{\alpha} \tilde{f}_i^{n+1},\quad v>0,\\
&\hat{f}(x_{i+1},v)=(1-\tilde{\alpha})f^n_{i+1}+\tilde{\alpha} \tilde{f}_{i+1}^{n+1},\quad  v<0.
\end{cases}
\end{equation}
Here $\tilde{f}_i^{n+1}$ is determined by the prediction step. Then $\hat{f}^n_{i+1/2}$ in \eqref{numerical_scheme} is given by
the outflow of the above steady state problem in each cell such that \begin{equation}
  \hat{f}^n_{i+1/2}(v)=\hat{f}(x_{i+1},v),\quad  v>0\qquad
 \hat{f}^n_{i+1/2}(v)=\hat{f}(x_{i},v),\quad   v<0.
\label{flux_steady_state}
\end{equation}
\bigskip

First of all, when $\epsilon$ is at $O(1)$ or $\Delta x,\Delta t\ll\varepsilon$, i.e. $\tilde{\alpha}=\Delta t/\varepsilon$, we show that \eqref{numerical_scheme} is a consist discretization for \eqref{eq:neutrontrans}.
We only consider the case when $v>0$. Since $\hat{f}^n_{i-1/2}$ is obtained from the steady problem step, it can be approximated by 
$$\begin{aligned}
&(1-\Delta t/\varepsilon)f_{i-1}^n+\Delta t/\varepsilon\tilde{f}^{n+1}_{i-1}+\Delta x\partial_x\hat{f}^n(x_{i-1},v)\\=
&(1-\Delta t/\varepsilon)f_{i-1}^n+\Delta t/\varepsilon\tilde{f}^{n+1}_{i-1}+\frac{\Delta x}{v}\Big(\frac{\sigma_T}{\varepsilon}\left(\hat{\rho}^n_{i-1}-\hat{f}^n_{i-1}\right)-\varepsilon\sigma_a\hat{\rho}^n_{i-1}+\varepsilon q\Big)\\
=&f_{i-1}^n+\frac{\Delta x}{v}\Big(\frac{\sigma_T}{\varepsilon}\left(\rho^n_{i-1}-f^n_{i-1}\right)-\varepsilon\sigma_a\rho^n_{i-1}+\varepsilon q\Big)+O(\Delta t^2)+\frac{\Delta x}{v}O(\Delta x).
\end{aligned}
$$
From the CFL condition, $\Delta t<\frac{\epsilon }{\max|v|}\Delta x$, \eqref{numerical_scheme} can then be approximated by 
$$
 \frac{f_i^{n+1}-f_i^n}{\Delta t}+\frac{v}{\varepsilon \Delta x}\left(f_i^n-f_{i-1}^{n}\right)=
 \frac{\sigma_T}{\varepsilon^2}\left(\rho_{i-1}^n-f_{i-1}^n\right)-\sigma_a\rho_{i-1}^n+q+O(\Delta x).
$$
The derivation for $v<0$ is the same.
\bigskip

Moreover, we have the following proposition:
\begin{proposition}
When AP schemes are used in the prediction step and steady problem, the numerical scheme given in \eqref{numerical_scheme} has both WB and AP properties.
\end{proposition}
\begin{proof}
\textbf{AP property: }When $\varepsilon$ tends to zero, since AP schemes are used for the 
 prediction step, we have \eqref{eq:tildefrho}.
From \eqref{eq:tildealpha}, when $\varepsilon$ is small, $\tilde{\alpha}=1$, then the boundary conditions
for the steady state problem \eqref{eq:hatf}
on each cell becomes
$$
\hat{f}(x_{i},v)=\tilde{\rho}_{i}^{n+1}+O(\varepsilon),\quad  v>0,\qquad
\hat{f}(x_{i+1},v)=\tilde{\rho}_{i+1}^{n+1}+O(\varepsilon),\quad v<0.
$$
When $\varepsilon$ is small, at the leading order, the boundary conditions for the steady state problem \eqref{eq:hatf} 
are isotropic. Therefore we have
$$\hat{f}=\hat{\rho}+O(\varepsilon)$$
and
$$\begin{aligned}
 \hat{f}^n_{i+1/2}&=\hat{f}(x_{i+1},v)=\tilde{\rho}^{n+1}_{i+1}+O(\varepsilon),\quad  v>0,\\
 \hat{f}^n_{i+1/2}&=\hat{f}(x_{i},v)=\tilde{\rho}^{n+1}_{i}+O(\varepsilon),\quad   v<0.
 \end{aligned}
 $$
Then from \eqref{numerical_scheme}, 
$$\begin{aligned}
f_i^{n+1}&=\frac{\hat{f}_{i-\frac{1}{2}}+\frac{\varepsilon \Delta x}{v\Delta t}f_i^n}{1+\frac{\varepsilon \Delta x}{v\Delta t}}=\hat{f}_{i-\frac{1}{2}}^{n}+O(\varepsilon)=\tilde{\rho}^{n+1}_{i}+O(\varepsilon),\quad v>v_0,\\
f_i^{n+1}&=\frac{\hat{f}_{i+\frac{1}{2}}-\frac{\varepsilon \Delta x}{v\Delta t}f_i^n}{1-\frac{\varepsilon \Delta x}{v\Delta t}}=\hat{f}_{i+\frac{1}{2}}^{n}+O(\varepsilon)=\tilde{\rho}^{n+1}_{i}+O(\varepsilon),\quad v<-v_0.
\end{aligned}
$$with $v_0$ small positive value away from $0$. 
Therefore, we get
\begin{equation*}
\rho_i^{n+1}=\tilde{\rho}_i^{n+1}+O(\varepsilon).
\end{equation*}
Since $\tilde{\rho}_i^{n+1}$ satisfy the macroscopic equation at the discrete level, i.e. the leading order of $f_i^{n+1}$ evolves according to a discrete diffusion equation, which gives the AP property of the proposed scheme \eqref{numerical_scheme}.

\bigskip
\textbf{WB property: }To prove the WB property, we 
assume that at time $t^n$, $f_i^{n}=\bar{f}(x_i,v)$ where $\bar{f}(x,v)$ is the solution to the steady state equation 
 \eqref{eq:hatf} with given inflow boundary conditions on the whole computational domain. We need to show that when
 $\Delta x$, $\Delta t$ are small enough, $f_i^{n+1}=f_i^n+O(\Delta x\Delta t)$. For the resolved case under consideration, 
 $$\tilde{\alpha}=\frac{\Delta t}{\varepsilon}.$$
 Since the scheme at the prediction step is consistent, consider the simplest first order method, we have 
$$
\tilde{f}_i^{n+1}=f_i^n+O(\Delta x, \Delta t).
$$
The inflow boundary conditions in \eqref{eq:stableinflow} becomes
$$
\begin{cases}
&\hat{f}(x_i,v)=(1-\tilde{\alpha})f_i^n+\tilde{\alpha} \tilde{f}_i^{n+1}=f_i^n+O(\Delta x\Delta t),\quad v>0,\\
&\hat{f}(x_{i+1},v)=(1-\tilde{\alpha})f^n_{i+1}+\tilde{\alpha} \tilde{f}_{i+1}^{n+1}=f_{i+1}^n+O(\Delta x\Delta t),\quad  v<0.
\end{cases}
$$
Since the steady state problem is solved in a cell of size $\Delta x$, by the simplest Taylor expansion, the out flow can be approximated by at least second order
accuracy, so that
$$\begin{aligned}
& \hat{f}^n_{i+1/2}=\bar{f}(x_{i+1},v)+O(\Delta x\Delta t,\Delta x^2),\quad  v>0;\\
 &\hat{f}^n_{i+1/2}=\bar{f}(x_{i},v)+O(\Delta x\Delta t,\Delta x^2),\quad   v<0.
\end{aligned}
$$
Thus, \eqref{numerical_scheme} yields
\begin{equation}\label{eq:finplus1}
\begin{aligned}
f_i^{n+1}&=\frac{\hat{f}_{i-\frac{1}{2}}+\frac{\varepsilon \Delta x}{v\Delta t}f_i^n}{1+\frac{\varepsilon \Delta x}{v\Delta t}}
=\bar{f}(x_i,v)+O(\Delta x\Delta t,\Delta x^2),\quad v>0,\\
f_i^{n+1}&=\frac{\hat{f}_{i+\frac{1}{2}}-\frac{\varepsilon \Delta x}{v\Delta t}f_i^n}{1-\frac{\varepsilon \Delta x}{v\Delta t}}
=\bar{f}(x_i,v)+O(\Delta x\Delta t,\Delta x^2),\quad v<0.
\end{aligned}
\end{equation}
Due to the CFL condition, $O(\Delta t)= O(\Delta x)$, the solution approximate the steady state $\bar{f}$ with
a formally second order accuracy $O(\Delta x^2)$.
\end{proof}
\begin{remark}
The framework in this section does not depend on specific AP methods for the time evolutionary problem and the steady
state problem, we can choose any scheme in the literature that has been proved to be AP for various kinetic models. 
\end{remark}
\begin{remark}
The accuracy of preserving the steady state problem can be improved by using higher order method in space and time,
or we can repeat once the steady problem step by replacing $\tilde{f}_i^{n+1}$ from the prediction step by $f_i^{n+1}$
obtained in \eqref{eq:finplus1}.
\end{remark}
\section{The chemotaxis kinetic model}
In this section, we apply the framework in section 2 to the chemotaxis kinetic model \eqref{eq_chemotaxis_eps}.
Two specific AP schemes are respectively chosen for the time evolutionary problem and steady state problem. 
We first extend the unified gas kinetic scheme (UGKS) in \cite{Mieussens,Xu2,Xu} to get an AP scheme for the time evolutionary problem, then solve the steady state equation by extending the scheme in \cite{Min} which was originally designed for the isotropic neutron transport equation. In \cite{Bokai}, the authors have proposed AP schemes for
the chemotaxis kinetic model, our approach is different. The details are described below.

As in \cite{Mieussens}, the UGKS is a finite volume approach of integrating \eqref{eq_chemotaxis_eps} over $[x_{i-1/2},x_{i+1/2}] \times [t^n,t^{n+1}] \times V$. Let
$$f_i^n=\frac{1}{\Delta x}\int_{x_{i-1/2}}^{x_{i+1/2}}f(x,v,t^n)\,d x,\qquad\rho_i^n=\frac{1}{|V|}\int_Vf_i^n\,d v.$$
 $\rho_i^{n+1}$ and $f_i^{n+1}$ are updated as follows
\begin{equation}\label{eq_rho_tilde}
\frac{{\rho}_i^{n+1}-\rho_i^n}{\Delta t}+\frac{F_{i+1/2}^n-F_{i-1/2}^n}{\Delta x}=0,
\end{equation} 
\begin{multline}\label{eq_f_tilde}
 \frac{f_i^{n+1}-f_i^n}{\Delta t}+\frac{\Phi_{i+1/2}^n-\Phi_{i-1/2}^n}{\Delta x}=\frac{1}{\varepsilon^2}\left(\rho_i^{n+1}-f_i^{n+1}\right)\\
 +\frac{1}{\varepsilon}\left(\frac{1}{\abs{V}}\int_{V} \phi(v\sigma_{i+1/2})f_i^n-\phi(v\sigma_{i+1/2})f_i^n \right).
\end{multline}
Here the numerical flux \begin{equation}\label{eq:PhiF}\begin{aligned}&\Phi^n_{i+1/2}\approx \frac{1}{\varepsilon \Delta t}\int_{t^n}^{t^{n+1}}
vf(x_{i+1/2},v,t)\,dt,\\& F_{i+1/2}^n\approx\frac{1}{|V|}\int_V\Big(\frac{1}{\varepsilon \Delta t}\int_{t^n}^{t^{n+1}}
vf(x_{i+1/2},v,t)\,dt\Big)\,dv.\end{aligned}\end{equation} 
Since one dimensional problem is considered here, we use discrete ordinate method for the velocity discretization.
The most crucial step for UGKS is to determine $\Phi^n_{i+1/2}$
and $F_{i+1/2}^n$.  The details are as below:

\subsection{Determine $\Phi^n_{i+1/2}$, $F_{i+1/2}^n$}
In one dimensional space, the chemotaxis model \eqref{eq_chemotaxis_eps} can be rewritten as follows:
\begin{equation}\label{eq_kinetics}
\partial_t f^{\varepsilon}+\frac{1+\varepsilon \phi(v\partial_x S^{\varepsilon})}{\varepsilon^2}f^{\varepsilon}+\frac{v}{\varepsilon} \partial_x f^{\varepsilon}=\frac{1}{\varepsilon^2} \mathcal{T}^1 f,
\end{equation}
where $\displaystyle (\mathcal{T}^1 f)(x,t):=\frac{1}{\abs{V}}\int_{V}\big(1+\varepsilon \phi(v\partial_x S)\big)f(x,v,t)dv$.

To approximate the flux $F_{i+1/2}^n$ in \eqref{eq:PhiF},
we first approximate $\partial_x S$ by a piecewise constant function such that
\begin{equation*}
\partial_x S \approx \partial_x S(x_{i+1/2}):=\sigma_{i+1/2},\quad x \in [x_{i},x_{i+1}).
\end{equation*}
It is important to note that $\sigma_{i+1/2}$ approximate $\partial_xS$ in the interval $[x_{i},x_{i+1})$
while $f_i^n$ is the density average over the cell $[x_{i-1/2},x_{i+1/2})$. This choice is important to preserve the 
advection term in the limiting Keller-Segel model when $\varepsilon$ becomes small.

On the interval $[x_{i},x_{i+1})$, multiplying both sides of \eqref{eq_kinetics} by $\exp{\left(\frac{(1+\varepsilon \phi(v\sigma_{i+1/2})}{\varepsilon^2}t\right)}$ yields 
\begin{equation*}
\frac{d}{dt} \left[f(x+\frac{v}{\varepsilon}t,v,t) \exp{\left(\frac{(1+\varepsilon \phi(v\sigma_{i+1/2})}{\varepsilon^2}t\right)}\right]=\frac{\mathcal{T}^1 f(x,t)}{\varepsilon^2} \exp{\left(\frac{(1+\varepsilon \phi(v\sigma_{i+1/2})}{\varepsilon^2}t\right)}.
\end{equation*}
Integrating the above equation over $(t^n,t)$ yields
\begin{equation}\label{eq:fxiplus}\begin{aligned}
&f(x_{i+1/2},v,t)=f(x_{i+1/2}-\frac{v}{\varepsilon}(t-t^n),v,t^n)\exp{\left(-\frac{(1+\varepsilon \phi(v\sigma_{i+1/2})}{\varepsilon^2}(t-t^n)\right)}\\
&\qquad+\frac{1}{\varepsilon^2}\int_{t^n}^t \mathcal{T}^1 f(x_{i+1/2}-\frac{v}{\varepsilon}(t-s),s)\exp{\left(-\frac{(1+\varepsilon \phi(v\sigma_{i+1/2})}{\varepsilon^2}(t-s)\right)}ds.
\end{aligned}
\end{equation}
This is an exact expression for $f(x_{i+1/2},v,t)$ that will be used to determine $\Phi^n_{i+1/2}$, $F_{i+1/2}^n$ in \eqref{eq:PhiF}.
At this stage, we need to approximate $f(x,v,t^n)$ and $\mathcal{T}^1 f(x,t)$ on the right hand side of \eqref{eq:fxiplus}.
 $f$ is approximated by a piecewise constant function and $\mathcal{T}^1 f$ by a piecewise linear function as follows:
\begin{equation*}
\begin{aligned}
& f(x,v,t^n)=\left\{
\begin{aligned}
& f_i^n,\quad x<x_{i+1/2},\\
& f_{i+1}^n,\quad x>x_{i+1/2},
\end{aligned}
\right.\\
&\mathcal{T}^{1} f(x,t)=\left\{
\begin{aligned}
&\mathcal{T}^1 f^n_{i+1/2}+\delta^L\mathcal{T}^{1}f^n_{i+1/2}(x-x_{i+1/2}) ,\quad x<x_{i+1/2},\\
&\mathcal{T}^1 f^n_{i+1/2}+\delta^R\mathcal{T}^{1}f^n_{i+1/2}(x-x_{i+1/2}),\quad x>x_{i+1/2}.
\end{aligned}
\right.
\end{aligned}
\end{equation*}
Here $\mathcal{T}^1 f^n_{i+1/2},\,\delta^L\mathcal{T}^1f^n_{i+1/2},\, \delta^R\mathcal{T}^1f^n_{i+1/2}$ are defined by: 
\begin{equation*}
\left\{
\begin{aligned}
\mathcal{T}^1 f^n_{i+1/2}&:=\frac{1}{\abs{V}}\int_{V^{-}} (1+\varepsilon \phi(v\sigma_{i+1/2}))f_{i+1}^n+\frac{1}{\abs{V}}\int_{V^+} (1+\varepsilon \phi(v\sigma_{i+1/2}))f_i^n,\\
\delta^L\mathcal{T}^1f^n_{i+1/2}&:=\frac{\mathcal{T}^1f^n_{i+1/2}-\mathcal{T}^1f^n_{i}}{\Delta x/2},\\
\delta^R\mathcal{T}^1f^n_{i+1/2}&:=\frac{\mathcal{T}^1f^n_{i+1}-\mathcal{T}^1f^n_{i+1/2}}{\Delta x/2},
\end{aligned}
\right.
\end{equation*}
with $V^+=V \cap \mathbb{R}^+$ and $V^-=V\cap \mathbb{R}^-$.\\
Substituting the above approximations into \eqref{eq:fxiplus} yields an expression for $f(x_{i+1/2},v,t)$ such that:\\
For $v>0$,\begin{equation}\label{eq_fvp}
\begin{aligned}
&f(x_{i+1/2},v,t)=f_i^n \exp{\left(-\frac{(1+\varepsilon \phi(v\sigma_{i+1/2})}{\varepsilon^2}(t-t^n)\right)}\\
&\quad+\frac{\mathcal{T}^1 f^n_{i+1/2}}{1+\varepsilon \phi(v\sigma_{i+1/2})}\left(1-\exp{\left(-\frac{(1+\varepsilon \phi(v\sigma_{i+1/2})}{\varepsilon^2}(t-t^n)\right)} \right)
+v \varepsilon \frac{\delta^L\mathcal{T}^1f^n_{i+1/2}}{(1+\varepsilon \phi(v\sigma_{i+1/2}))^2}\\
&\quad\cdot\left[\left(1+ \frac{1+\varepsilon \phi(v\sigma_{i+1/2})}{\varepsilon^2} (t-t^n)\right)\exp{\left(-\frac{(1+\varepsilon \phi(v\sigma_{i+1/2})}{\varepsilon^2}(t-t^n)\right)}-1 \right],
\end{aligned}
\end{equation}
and for $v<0$, 
\begin{equation}\label{eq_fvm}
\begin{aligned}&f(x_{i+1/2},v,t)
=f_{i+1}^n \exp{\left(-\frac{(1+\varepsilon \phi(v\sigma_{i+1/2})}{\varepsilon^2}(t-t^n)\right)}\\
&\quad+\frac{\mathcal{T}^1 f^n_{i+1/2}}{1+\varepsilon \phi(v\sigma_{i+1/2})}\left(1-\exp{\left(-\frac{(1+\varepsilon \phi(v\sigma_{i+1/2})}{\varepsilon^2}(t-t^n)\right)} \right)
+v \varepsilon \frac{\delta^R\mathcal{T}^1f^n_{i+1/2}}{(1+\varepsilon \phi(v\sigma_{i+1/2}))^2}\\
&\quad\cdot\left[\left(1+ \frac{1+\varepsilon \phi(v\sigma_{i+1/2})}{\varepsilon^2} (t-t^n)\right)\exp{\left(-\frac{(1+\varepsilon \phi(v\sigma_{i+1/2})}{\varepsilon^2}(t-t^n)\right)}-1 \right].
\end{aligned}
\end{equation}
Then the flux $\Phi^n_{i+1/2}(v)$ in \eqref{eq:PhiF} can be approximated by
\begin{equation}\label{eq_phi}
 \begin{aligned}
 &\Phi_{i+1/2}(v)=A v f_{i+1}^n+Bv \mathcal{T}^1 f^n_{i+1/2}+Cv^2 \delta^R\mathcal{T}^1 f^n_{i+1/2},\quad \text{for } v<0,\\
 & \Phi_{i+1/2}(v)=A v f_i^n+Bv \mathcal{T}^1 f^n_{i+1/2}+Cv^2 \delta^L\mathcal{T}^1 f^n_{i+1/2},\quad \text{for } v>0,\\
 \end{aligned}
\end{equation}
where the coefficients $A(v,\varepsilon,\Delta t),B(v,\varepsilon,\Delta t),C(v,\varepsilon,\Delta t)$ can be determined explicitly such that
\begin{equation}\label{eq_A_B_C}
\begin{aligned}
A(v,\varepsilon,\Delta t):&= \frac{\varepsilon}{\Delta t\big(1+\varepsilon \phi(v\sigma_{i+1/2})\big)}\left(1-\exp{\big(-\frac{1+\varepsilon \phi(v\sigma_{i+1/2})}{\varepsilon^2}\Delta t\big)}\right),\\
B(v,\varepsilon,\Delta t):&=\frac{1}{\varepsilon(1+\varepsilon \phi(v\sigma_{i+1/2}))}\\
&-\frac{\varepsilon}{\Delta t(1+\varepsilon \phi(v\sigma_{i+1/2}))^2}\left(1-\exp{\big(-\frac{1+\varepsilon \phi(v\sigma_{i+1/2})}{\varepsilon^2}\Delta t\big)}\right),\\
C(v,\varepsilon,\Delta t):&=\frac{2\varepsilon^2}{\Delta t (1+\varepsilon \phi(v\sigma_{i+1/2}))^3}\left(1-\exp{\big(-\frac{1+\varepsilon \phi(v\sigma_{i+1/2})}{\varepsilon^2}\Delta t\big)}\right)\\
\phantom{C(v,\varepsilon,\Delta t)}&-\frac{1}{(1+\varepsilon \phi(v\sigma_{i+1/2}))^2}\left(1+\exp{\big(-\frac{1+\varepsilon \phi(v\sigma_{i+1/2})}{\varepsilon^2}\Delta t\big)}\right).
\end{aligned}
\end{equation}
Furthermore, $F_{i+1/2}^n$ in \eqref{eq:PhiF} is given by
\begin{multline}\label{eq_flux}
F_{i+1/2}^n=\frac{1}{\abs{V}} \int_{V^-} A v f_{i+1}^n dv +\frac{1}{\abs{V}} \int_{V^+}A v f_i^n dv+\frac{1}{\abs{V}}\mathcal{T}^1 f^n_{i+1/2}\int_{V} vB dv\\
+\frac{1}{\abs{V}}\delta^R\mathcal{T}^1 f^n_{i+1/2} \int_{V^-}C v^2 dv +\frac{1}{\abs{V}} \delta^L\mathcal{T}^1 f^n_{i+1/2}\int_{V^+} C v^2 dv.
\end{multline}
and we complete the construction of the scheme. We can prove that the scheme is AP by employing similar approach as in  
\cite{Mieussens}.
\subsection{AP property}
In this part, we give a formal derivation of the AP property for the UGKS proposed in \eqref{eq_rho_tilde}--\eqref{eq_f_tilde}.
When $\varepsilon$ goes to zero, asymptotic expansions of $A,B,C$ given in \eqref{eq_A_B_C} read
 $$
 A=O(\varepsilon),\qquad
 B=\frac{1}{\varepsilon}-\phi(v\sigma_{i+1/2})+O(\varepsilon),\qquad
 C=-1+O(\varepsilon).
 $$
The leading order term of \eqref{eq_f_tilde} yields $f_i^{n+1}=\rho_i^{n+1}+O(\varepsilon)$ and we only need to show
that \eqref{eq_rho_tilde} satisfies the equation for $\rho$ in \eqref{limit_eq},
at the discrete level. 
Suppose that $f_i^n=\rho_i^n+O(\varepsilon)$, then 
\begin{equation*}
\left\{
\begin{aligned}
\mathcal{T}^1 f^n_{i+1/2}&=\frac{1}{2}\left(\rho_i^n+\rho_{i+1}^n\right)+O(\varepsilon),\\
\delta^L\mathcal{T}^1f^n_{i+1/2}&=\frac{\rho^n_{i+1}-\rho_i^n}{\Delta x}+O(\varepsilon),\\
\delta^R\mathcal{T}^1f^n_{i+1/2}&=\frac{\rho^n_{i+1}-\rho_i^n}{\Delta x}+O(\varepsilon).
\end{aligned}
\right.
\end{equation*}
We deduce that the expansion of $F_{i+1/2}^n$ reads:
\begin{equation*}
 F_{i+1/2}^n=-\frac{\rho_i^n+\rho_{i+1}^n}{2\abs{V}}\left(\int_{V} v\phi(v\sigma_{i+1/2})dv \right)-\frac{\rho_{i+1}^n-\rho_i^n}{3 \Delta x}+O(\varepsilon).
\end{equation*}
Therefore,
$$\begin{aligned}
 &\frac{F^n_{i+1/2}-F^n_{i-1/2}}{\Delta x}\\
 =&-\frac{\rho_{i+1}^n-2\rho_i^n+\rho_{i-1}^n}{3(\Delta x)^2}
 +\Big(-\Big(\frac{1}{\abs{V}}\int_{V} v\phi(v\sigma_{i+1/2})dv \Big) \frac{\rho_i^n
 +\rho_{i+1}^n}{2}\\&\qquad\qquad\qquad\qquad\qquad\quad+\Big(\frac{1}{\abs{V}}\int_{V} v\phi(v\sigma_{i-1/2})dv \Big) \frac{\rho_i^n+\rho_{i-1}^n}{2} \Big)+O(\varepsilon).
\end{aligned}
$$
 In the limit of $\varepsilon\to 0$, the discretization \eqref{eq_rho_tilde} becomes
\begin{multline*}
\frac{\rho_i^{n+1}-\rho_i^n}{\Delta t}=\frac{\rho_{i+1}^n-2\rho_i^n+\rho_{i-1}^n}{3(\Delta x)^2}\\
 +\left(\frac{1}{\abs{V}}\left(\int_{V} v\phi(v\sigma_{i+1/2})dv \right) \frac{\rho_i^n+\rho_{i+1}^n}{2}-\frac{1}{\abs{V}}\left(\int_{V} v\phi(v\sigma_{i-1/2})dv \right) \frac{\rho_i^n+\rho_{i-1}^n}{2} \right).
\end{multline*}
which is a consistent discretization of the equation for $\rho$ in \eqref{limit_eq}.
The proposed scheme is AP.

\subsection{Steady state problem for the chemotaxis kinetic model}
To solve
the steady state problem, we start from the most used discrete ordinate method \cite{Chandrasekhar,Kubelka}.
The discrete ordinate method is to discretize the velocity space $V=[-1,1]$ by 
a quadrature set $\{\mu_m,\omega_m\}$ ($m\in V_m=\{-N,\cdots,-1,1,\cdots,N\}$). In one dimensional case, 
 the most used and well accepted quadrature is the Gaussian quadrature. 
In order to preserve the diffusion limit equation, the points $\mu_m$ and weights $\omega_m$ satisfy
\begin{equation*}
\begin{aligned}
& \mu_{-m}=-\mu_m,\quad \omega_{-m}=\omega_m,\quad m\in 1\cdots N,\\
& \sum_{m} \omega_m =2,\quad \sum_{m} \omega_m \mu_m^2=\frac{2}{3}.\\
\end{aligned}
\end{equation*}
 The discrete ordinate method for the steady state chemotaxis kinetic model on each cell $[x_i,x_{i+1})$ writes
\begin{equation}
\mu_m \partial_x f_m=\frac{1}{2\varepsilon}\sum_{n \in V}\omega_n\left(1+\varepsilon \phi(\mu_n \sigma_{i+1/2})\right)f_n-\frac{1}{\varepsilon}\left(1+\varepsilon \phi(\mu_m \sigma_{i+1/2})\right)f_m,\quad m\in V_m.
\label{steady_state_chemotaxis}
\end{equation}
This is  a linear ODE system with constant matrix coefficient.  Together with the inflow boundary conditions, the exact solution 
can be obtained analytically. 
We use the following procedure to construct the general solution on each interval $[x_i,x_{i+1})$.
We seek eigenfunctions of the form \begin{equation}l_m \exp(-\frac{\zeta}{\varepsilon} x).\label{eq:lzeta}\end{equation} By substituting this form into \eqref{steady_state_chemotaxis}, one obtains 
\begin{equation}\label{eq:eigenphi}
 \left[\left(1+\varepsilon \phi(\mu_m \sigma_{i+1/2})\right)-\mu_m \zeta \right]l_m=\frac{1}{2}\sum_{n \in V}\omega_n\left(1+\varepsilon \phi(\mu_m \sigma_{i+1/2})\right)l_m.
\end{equation}
The above equation holds for all $m\in V_m$ which indicates that $\zeta$ is the eigenvalue of the matrix 
\begin{equation}
E=U^{-1}\big(I-\frac{1}{2}W\big)\big(I+\varepsilon\phi(U\sigma_{i+1/2})\big).
\label{matrix}
\end{equation}
Here $U=\mbox{Diag}(\mu_m)_{m\in V_m}$ and 
$\phi(U\sigma_{i+1/2})=\mbox{Diag}\big(\phi(\mu_m\sigma_{i+1/2})\big)_{m\in V_m}$ are diagonal matrixes, $W$ is a rank 1 matrix whose rows are $(\omega_m)_{m\in V_m}$. Due to \eqref{eq:eigenphi}, $\zeta$ is the root of 
\begin{equation}\label{eq:eigenzeta}
\frac{1}{2}\sum_{m\in V_m}\frac{\big(1+\varepsilon\phi(\mu_m\sigma_{i+1/2})\big)\omega_m}{1+\varepsilon\phi(\mu_m\sigma_{i+1/2})-\mu_m\zeta}=1
\end{equation}
and the
eigenvector associated to $\zeta$ can be given by 
\begin{equation}
l_m=\frac{1}{1+\varepsilon \phi(\mu_m \sigma_{i+1/2})-\mu_m \zeta},\qquad m\in V_m.
\label{eq_l_m}
\end{equation}
It is easy to check that when $\varepsilon\neq 0$, $\zeta=0$ is a simple root if $\phi(u)=-\phi(-u)$. Besides if
$\nu_m=\frac{1+\varepsilon \phi(\mu_m \sigma_{i+1/2})}{\mu_m}$ ($m\in V_m$) are different from each other, 
since the sign of $\nu_m$ are determined by $\mu_m$, the functional $\frac{1}{2}\sum_{m\in V_m}\frac{\big(1+\varepsilon\phi(\mu_m\sigma_{i+1/2})\big)\omega_m}{1+\varepsilon\phi(\mu_m\sigma_{i+1/2})-\mu_m\zeta}$
tends to plus or minus infinity when $\zeta$ approaches $\nu_m$ from left or right. Therefore, 
\eqref{eq:eigenzeta} has $2N$ simple roots when $\varepsilon$ is away from zero. However when $\varepsilon\to 0$, $E$ becomes $U^{-1}(I-W/2)$ whose eigenvalues admit a zero double root \cite{Chandrasekhar,Min}. The general solutions to \eqref{steady_state_chemotaxis} are different for these two different cases, we discuss them separately in the subsequent part.

\bigskip
\textbf{Case I: } If \eqref{eq:eigenzeta} has $2N$ different simple roots, we can find $2N$ linearly independent eigenfunctions of the form $l_m \exp(-\frac{\zeta}{\varepsilon} x)$. Then the general solution to \eqref{steady_state_chemotaxis} can be written as
$$f_m(x)=
 \sum_{\zeta_n<0} B^n l_m^n\exp\big(-\frac{\zeta_n}{\varepsilon}(x-x_{i+1})\big)
+\sum_{\zeta_n>0} A^n l_m^n\exp\big(-\frac{\zeta_n}{\varepsilon}(x-x_{i})\big),
$$
The coefficients $B^n,A^n$ can be determined by
the inflow boundary conditions
\begin{equation*}
\left\{
\begin{aligned}
& f_m(x_{i+1})=f_m^R,\quad m<0,\\
& f_m(x_i)=f_m^L,\quad m>0.
\end{aligned}
\right.
\end{equation*}
Then the outflows are given by
\begin{equation*}
 \begin{aligned}
  f_m(x_i)&=\sum_{\zeta_n<0} B^n l_m^n\exp\big(\frac{\zeta_n}{\varepsilon}\Delta x\big)+\sum_{\zeta_n>0} A^n l_m^n,\quad m<0,\\
  f_{m}(x_{i+1})&=\sum_{\zeta_n<0} B^n l_m^n+\sum_{\zeta_n>0} A^n l_m^n\exp\big(-\frac{\zeta_n}{\varepsilon}\Delta x\big),\quad m>0.
 \end{aligned}
\end{equation*}

\bigskip
\textbf{Case II:}
We consider the limiting case when $E=U^{-1}(I-W/2)$, the eigenvalues for this matrix have been proved in \cite{Min}
to have the following propperty
\begin{theorem}
The equation $$\frac{1}{2}\sum_{m\in V_m}\frac{\omega_m}{1-\mu_m\zeta}=1$$ has $2(N-1)$ simple roots appear in positive/negative paris while $0$ is a double root.
 \end{theorem}
We set the double zero eigenvalue $\zeta_N$ and $\zeta_{-N}$ and rearrange eigenvalues $\zeta_n$ from the lowest to the highest.
\begin{equation*}
\zeta_{1-N}<\cdots<\zeta_{-1}<0<\zeta_1<\cdots<\zeta_{N-1}.
\end{equation*}
For eigenvalues different of zero, eigenfunctions are exponential functions as in \eqref{eq:lzeta}. Let us find eigenfunctions corresponding to the zero eigenvalue. By using the same token as in \cite{Min}, we are looking for 
solution of the form $\alpha_m x+\beta_m$.

Injecting $\alpha_m x+\beta_m$ into \eqref{steady_state_chemotaxis} gives two solutions
\begin{equation*}
\frac{1}{1+\varepsilon \phi(\mu_m \sigma_{i+1/2})},\qquad
\frac{x}{1+\varepsilon \phi(\mu_m \sigma_{i+1/2})}-\frac{\varepsilon \mu_m}{\big(1+\varepsilon \phi(\mu_m \sigma_{i+1/2})\big)^2}.
\end{equation*}
Finally the general solution of \eqref{steady_state_chemotaxis} writes
\begin{equation*}
\begin{aligned}
f_m(x)=&\frac{c_{N}}{1+\varepsilon \phi(\mu_m \sigma_{i+1/2})}  +\frac{c_{-N}}{1+\varepsilon \phi(\mu_m \sigma_{i+1/2})} \big(x-\frac{\varepsilon\mu_m }{1+\varepsilon \phi(\mu_m \sigma_{i+1/2})}\big)\\
&+ \sum_{1\leq \abs{n} \leq N-1} c_n l_m^n\exp\left(-\frac{\zeta_n}{\varepsilon}x\right),
\end{aligned}
\end{equation*}
where $c_n$ ($1\leq |n|\leq N$) are constants to be determined by the boundary conditions.
Since $\varepsilon$ can become very small, overflow may occur when evaluating $\exp\big(-\xi_nx/\varepsilon\big)$.
In order to have $c_n$ at the same scale as $f_m(x)$, we rewrite the general solution in the interval $(x_i,x_{i+1})$
into\begin{equation}
\begin{aligned}f_m(x)=&\frac{c_1}{1+\varepsilon \phi(\mu_m \sigma_{i+1/2})}  +\frac{c_2}{1+\varepsilon \phi(\mu_m \sigma_{i+1/2})} \Big((x-x_i)-\frac{\mu_m \varepsilon}{1+\varepsilon \phi(\mu_m \sigma_{i+1/2})}\Big)\\
&+ \sum_{n<0} B^n l_m^n\exp\big(-\frac{\zeta_n}{\varepsilon}(x-x_{i+1})\big)
+\sum_{n>0} A^n l_m^n\exp\big(-\frac{\zeta_n}{\varepsilon}(x-x_{i})\big),
\end{aligned}
\label{eq_fm}
\end{equation}
with $c_1,c_2,B^n,A^n$ constants. Then we can determine $c_1,c_2,B^n,A^n$ by the inflow boundary condition, so that the outflow.

\section{The radiative transport equation}
For the radiative transport equation \eqref{radiative_transport}, the same as in the previous section, we can choose two specific AP schemes respectively for the time evolutionary problem and steady state problem. For the time evolutionary problem, 
the UGKS scheme proposed in \cite{Sun} is employed, while similar discretization as in subsection 3.3 can be used to solve the steady state problem.
We sketch the discretization used in \cite{Sun} in the subsequent part and for more details and the proof of AP, one can refer to \cite{Sun}.

\subsection{The AP UGKS for the grey radiative transport equation}
The idea is to first introduce $\psi=acT^4$ and rewrite the system \eqref{radiative_transport} into
\begin{equation}
\left\{
\begin{aligned}
& \frac{1}{c} \partial_t I+\frac{1}{\varepsilon} v \partial_x I=\frac{\sigma}{\varepsilon^2} \left(\frac{1}{|V|}\psi-I\right)+ q(x,v),\\
& \varepsilon^2 \partial_t \psi=\beta(\psi) \sigma(\rho-\psi),
\end{aligned}
\right.
\label{radiative_transportph}
\end{equation}
where $\rho:=\int_{V} I dv$ is the total radiation intensity and $\beta$ is a function of $\psi$ given by
$$
\beta(\psi)=\frac{4 ac}{C_v}\left(\frac{\psi}{a c}\right)^{3/4}.
$$
Then we take the integral with respect to $v$ in \eqref{radiative_transportph} and get a system for $\rho$, $\psi$,
with an advection term depending on $\int_VvI\,dv$. Similar  as the UGKS for the chemotaxis model, the first step is 
to get a prediction for $\rho^{n+1}$ by
\begin{equation}
\left\{
\begin{aligned}
& \rho_{i}^{n+1}=\rho_i^n+\frac{\Delta t}{\Delta x}\left(F_{i-1/2}^n-F_{i+1/2}^n \right)+\frac{\sigma c\Delta t}{\varepsilon^2}(\psi_i^{n+1}-\rho_i^{n+1})+c\Delta t\int_V q_i dv,\\
& \psi_i^{n+1}=\psi_i^n+\frac{\sigma \beta_i^{n+1}\Delta t}{\varepsilon^2}\left(\rho_i^{n+1}-\psi_i^{n+1}\right),
\end{aligned}
\right.
\label{update_rho}
\end{equation}
where $$F_{i+1/2}^n:=\frac{c}{|V|}\int_V\Big(\frac{1}{\varepsilon \Delta t}\int_{t^n}^{t^{n+1}}
vI(x_{i+1/2},v,t)\,dt\Big)\,dv,\quad \displaystyle q_i(v):=\frac{1}{\Delta x} \int_{x_{i-1/2}}^{x_{i+1/2}} q(x,v)dx.$$ 
After discretizing the velocity space by discrete-ordinate method, the numerical flux $F_{i-1/2}^n$ and source term are approximated by

\begin{equation*}
\begin{aligned}
F_{i-1/2}^n&\approx A\sum_{m=-N}^N \omega_m \mu_m(I_{i-1/2,m}^{n-} \mathbf{1}_{\mu_m>0}+I_{i-1/2,m}^{n+} \mathbf{1}_{\mu_m<0})+\frac{2D}{3\Delta x}(\psi_i^{n+1}-\psi_{i-1}^{n+1})\\
+&B\sum_{m=-N}^N\omega_m \mu_m^2(\delta_x I_{i-1,m}^n \mathbf{1}_{\mu_m>0}+\delta_x I_{i,m}^n \mathbf{1}_{\mu_m<0})+\abs{V}\varepsilon^2 \frac{C}{\sigma}\sum_{m\in V_m} \omega_m \mu_m q_{i,m},\\
\int_V q_i dv&\approx \sum_{m\in V_m}\omega_m q_{i,m}=\sum_{m\in V_m}\omega_m q_i(\mu_m).
\end{aligned}
\end{equation*}
Here $\mathbf{1}_{\mu_m<0}$, $\mathbf{1}_{\mu_m>0}$ are the characteristic functions, 
$\delta_x I_{i,m}^n$ are the approximations to $\partial_xI(x_i,\mu_m,t^n)$ by the slope limiters   and $I_{i-1/2,m}^{n\pm}$ are given by 
\begin{equation*}
I_{i-1/2,m}^{n-}=I_{i-1,m}^n+\frac{\Delta x}{2}\delta_x I_{i-1,m}^n,\quad I_{i-1/2,m}^{n+}=I_{i,m}^n-\frac{\Delta x}{2}\delta_x I_{i,m}^n.
\end{equation*}
The coefficients $A,B,C,D$ depend on $\varepsilon$ and $\Delta t$, which will be specified later.

After obtaining the macroscopic variables $\rho_i^{n+1},\psi_i^{n+1}$ from \eqref{update_rho}, $I_{i,m}^{n+1}$ are computed by

\begin{equation}
\begin{aligned}
  I_{i,m}^{n+1}&=I_{i,m}^n+\frac{\Delta t}{\Delta x}\left(\Phi^n_{i-1/2,m}-\Phi^n_{i+1/2,m}\right)+\frac{\sigma c\Delta t}{\varepsilon^2}\left(\frac{\psi_i^{n+1}}{|V|}-I_{i,m}^{n+1}\right)+c\Delta t q_{i,m},
\end{aligned}
\label{update_I}
\end{equation}

where $\Phi^n_{i-1/2,m}$ are given by
\begin{equation*}
\begin{aligned}
\Phi^n_{i-1/2,m}=& A \mu_m\left(I_{i-1/2,m}^{n-} \mathbf{1}_{\mu_m>0}+I_{i-1/2,m}^{n+} \mathbf{1}_{\mu_m<0}\right)+C\mu_m \psi_{i-1/2}^{n+1}\\
+& D \left(\mu_m^2 \delta_x \psi_{i-1/2}^{n+1,L}\mathbf{1}_{\mu_m>0}+\mu_m^2 \delta_x \psi_{i-1/2}^{n+1,R}\mathbf{1}_{\mu_m<0}\right)\\
+& B \left(\mu_m^2 \delta_x I_{i-1,m}^n \mathbf{1}_{\mu_m>0}+\mu_m^2 \delta_x I_{i,m}^n \mathbf{1}_{\mu_m<0}\right)\\
+& E \delta_t \psi_{i-1/2}^{n+1}
+ \frac{\abs{V}\varepsilon^2 C}{\sigma} \mu_m q_{i,m},\\
\end{aligned}
\end{equation*}

with $\psi_{i+1/2}^{n+1}$, $\delta_x \psi_{i-1/2}^{n+1,R}$, $\delta_x \psi_{i-1/2}^{n+1,L}$, $\delta_t \psi_{i-1/2}^{n+1}$ being defined by 
\begin{equation*}
\begin{aligned}
&\psi_{i+1/2}^{n+1}=\frac{1}{2}(\psi_i^{n+1}+\psi_{i+1}^{n+1}).
\qquad \delta_x \psi_{i-1/2}^{n+1,L}=\frac{\psi_{i-1/2}^{n+1}-\psi_{i-1}^{n+1}}{\Delta x/2},\\
& \delta_t \psi_{i-1/2}^{n+1}=\frac{\psi_{i-1/2}^{n+1}-\psi_{i-1/2}^{n}}{\Delta t}, \qquad \delta_x \psi_{i-1/2}^{n+1,R}=\frac{\psi_{i}^{n+1}-\psi_{i-1/2}^{n+1}}{\Delta x/2}.
\end{aligned}
\end{equation*}
$ \delta_x \psi_{i-1/2}^{n+1,L}$ and $\delta_x \psi_{i-1/2}^{n+1,R}$ are the same and we use different notations in order to be consistent with notations in \cite{Sun}.
In \cite{Sun}, different notations are used for $\delta_x \psi_{i-1/2}^{n+1,R}$ and $\delta_x \psi_{i-1/2}^{n+1,L}$ in order to consider different approximations for $\psi_{i-1/2}^{n+1}$. Some of them can lead to the case when $\delta_x \psi_{i-1/2}^{n+1,R}$ and $\delta_x \psi_{i-1/2}^{n+1,L}$ are different.

Let $\nu=\frac{c\sigma}{\varepsilon^2}$, the coefficients $A,B,C,D,E$ appeared in all above formula are given by
\begin{equation*}
\begin{aligned}
A &=\frac{c}{\varepsilon \Delta t \nu}\left(1-\exp{-(\nu \Delta t)}\right),\\
C &=\frac{c^2\sigma}{\abs{V}\varepsilon^3 \Delta t \nu}\left(\Delta t-\frac{1}{\nu}(1-\exp{(-\nu \Delta t)})\right),\\
D &=-\frac{c^3\sigma}{\abs{V}\varepsilon^4 \Delta t \nu^2}\left(\Delta t(1+\exp{(-\nu \Delta t)})-\frac{2}{\nu}(1-\exp{(-\nu \Delta t)})\right),\\
B&=-\frac{c^2}{\varepsilon^2\nu^2 \Delta t}\left(1-\exp{(-\nu \Delta t)}-\nu \Delta t \exp{(-\nu \Delta t)} \right),\\
E&=\frac{c^2\sigma}{\abs{V}\varepsilon^3\nu^3\Delta t}\left(1-\exp{(-\nu \Delta t)}-\nu \Delta t \exp{(-\nu \Delta t)}-\frac{1}{2} (\nu\Delta t)^2\right).
\end{aligned}
\end{equation*}

\subsection{Steady states for the radiative transport equation}
We use the discrete ordinate method similar as in subsection 3.3 and refer to \cite{Gosse2} for more details.
The one dimensional discrete ordinate steady state problem for \eqref{radiative_transportph} reads
%
\begin{equation}
\mu_m \partial_x I_m=\frac{\sigma}{\varepsilon}(\sum_m \omega_m I_m-I_m)+\varepsilon q_m.
\label{steady_problem_app}
\end{equation}
We use the same AP scheme as in subsection 3.3 and the above equation falls into the category of Case II.
The only difference is that there exists a source term $q_m$, which can be easily built into the scheme by constructing an approximated 
particular solution. The general solution with source term $q_m$ can be given by the summation of the general solution of the homogeneous problem and
the approximated particular solution. Then the coefficients in front of the general solutions of the homogeneous equation can 
be determined by the continuity of solution at the cell edge. For more details, we refer to \cite{Min}.

\section{Numerical simulations}
In this section, we apply the well-balanced and AP scheme designed in Section 2 to two test cases: kinetic chemotaxis and radiative transport models.
\subsection{Chemotaxis model}
The first simulation concerns the chemotaxis model \eqref{eq_chemotaxis_eps}.
Simulations are set on $x \in [-1,1]$ and $v \in [-1,1]$. We take $N_v=32$ points for the velocity  and $N_x=500$ points for the space discretization.
Parameters in \eqref{eq_chemotaxis_eps} are chosen as in \cite{Gosse} such that
\begin{equation*}
 \chi_S=1,D=15,\beta=60,\alpha=3.
\end{equation*}
and $\phi$ is of the form 
\begin{equation*}
 \phi(x)=-\chi_S \tanh{\frac{x}{\delta}},\qquad \mbox{with $\delta=1$}.
\end{equation*}

We impose specular boundary conditions for $f^{\varepsilon}$ and Dirichlet conditions for $S^{\varepsilon}$.
The initial condition is composed of two bumps in $x$ located at $\pm 0.65$ such that
\begin{equation*}
 f^0(x,v)=5\exp{\left(-10(x-0.65)^2-10(x+0.65)^2\right)} \exp{\left(-20(v-0.5)^2-20(v+0.5)^2\right)}.
\end{equation*}
To ensure the stability of the numerical scheme, the time step $\Delta t$ is chosen as follows
\begin{equation*}
\Delta t=\left\{\begin{aligned}\Delta x^2,\quad &\mbox{for $\varepsilon<\Delta x^2$,}\\
\epsilon\Delta x,\quad&\mbox{else}.\end{aligned}\right.
\end{equation*}
The requirement for the time step can be improved if the limiting discretization is implicit, the strategy of constructing an implicit limiting discretization has been 
discussed in \cite{Sun}. The goal of our present paper is to illustrate the idea of the AP-WB framework and thus we keep using the simple scheme.
\subsubsection{AP property}
In order to verify the AP property of our scheme, 
 we show that the convergence order of the scheme is independent of $\varepsilon$. In figure \ref{fig:convergence1}, we plot 
$e_{\Delta x}:=\|f_{\Delta x}(t)-f_{2\Delta x}(t)\|_1$ with respect to 
$\Delta x$ for different values of $\varepsilon$ and uniform convergence with respect to $\epsilon$ can be observed numerically. This guarantees the AP property.
\begin{figure}
 \centering
 \includegraphics[width=10cm]{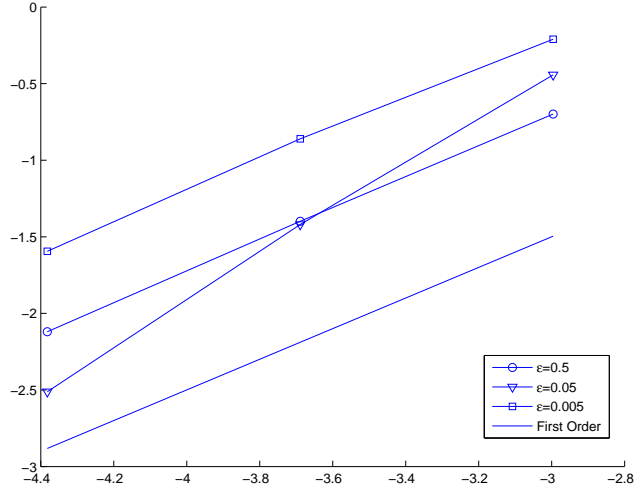}
 \caption{Convergence order with respect to $\Delta x$. Here the log-log plot of the mesh sizes with respect to the $L^1$ norm of the 
 probability density error in all directions is displayed. The mesh sizes are $1/20$, $1/40$, $1/80$, $1/160$ and 
 $e_{\Delta x}:=\max_{x\in(-1,1)}\|f_{\Delta x}(x,t)-f_{2\Delta x}(x,t)\|_1$ are plotted for different $\varepsilon$.}\label{fig:convergence1}
\end{figure}

\subsubsection{Model convergence in $\varepsilon$}
As predicted by the theoretical analysis,  when $\varepsilon \rightarrow 0$,
\begin{equation*}
f^{\varepsilon} \rightarrow \rho^0 \mathbf{1}_{v \in V},
\end{equation*}
where $\rho^0$ is the solution to the limiting Keller-Segel equation. 
To verify the above convergence, the total densities $\rho^\varepsilon$ are displayed in Figure~\ref{macro_rho_chemotaxis_AP}
for different values of $\varepsilon$ ranging from $10^{-1}$ 
to $10^{-6}$. As a comparison, we plot the solution of the limiting model in Figure~\ref{macro_rho_chemotaxis_AP}  and
can observe that $\rho^{\varepsilon}$ get closer to $\rho^0$ as $\varepsilon$ goes to zero.

The convergence order in $\varepsilon$ can be seen in Figure~\ref{log_error}, where we have plotted 
${\|f^{\varepsilon} (x,v)-\rho^{\varepsilon}(x) \mathbf{1}_{v \in V} \|}_2$ for different $\varepsilon$ in the logarithm scale. 
The numerical results indicate that the order of convergence is close to 1.
\begin{figure}[h]
\centering 
\includegraphics[width=10cm]{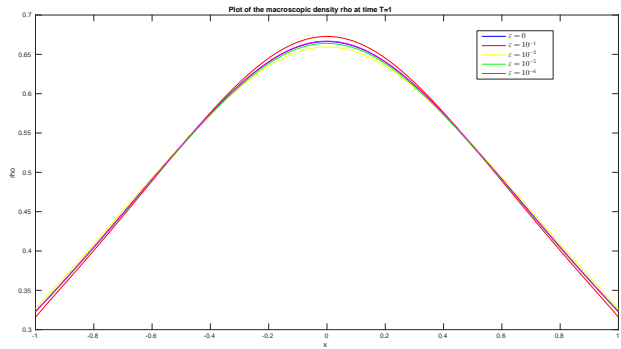}
\caption{Plot of the macroscopic density for different values of $\varepsilon$. $\varepsilon=0$ is the solution to the macroscopic Keller-Segel equation.}
\label{macro_rho_chemotaxis_AP}
\end{figure}

\begin{figure}[h]
\centering 
\includegraphics[width=10cm]{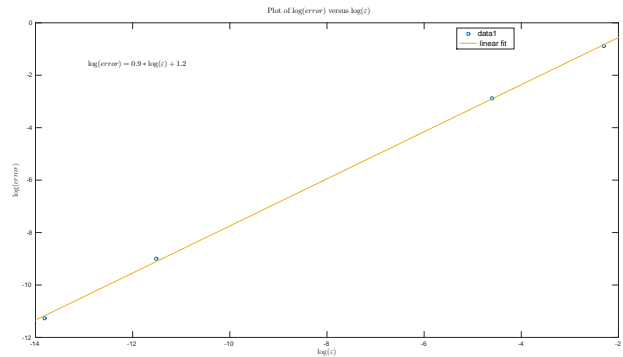}
\caption{Convergence order in $\varepsilon$ of the chemotaxi model \eqref{eq_chemotaxis_eps}.}
\label{log_error}
\end{figure}

\subsubsection{WB property}
In order to check the WB property, we calculate the long time behavior of the kinetic scheme for $\varepsilon=1$. The steady state is characterised by the vanishing of the macroscopic flux $J$. 
A way to measure whether the steady state is reached is to look at the evolution of the residues $r^n$ at a given time $t^n$ defined by 
\begin{equation*}
 r^n:=\left\|\sum_{m\in V} \omega_m \abs{f^{n+1}(\mu_m)-f^n(\mu_m)}\right\|_{2}.
\end{equation*}
Figure~\ref{residues_chemotaxis} shows the decrease of the residues and its stabilization when the steady state is reached. 
In addition, the flux $J$ is of order $\Delta x^2$ at this steady state as displayed in Figure~\ref{macroscopic_flux_chemotaxis}. 
\begin{figure}[h]
 \centering
    \includegraphics[width=10cm]{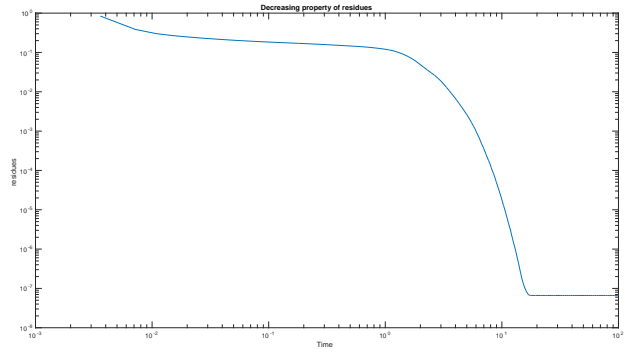}
    \caption{The evolution of residues with respect to time.}
    \label{residues_chemotaxis}
\end{figure}
\begin{figure}[h]
 \centering
  \includegraphics[width=10cm]{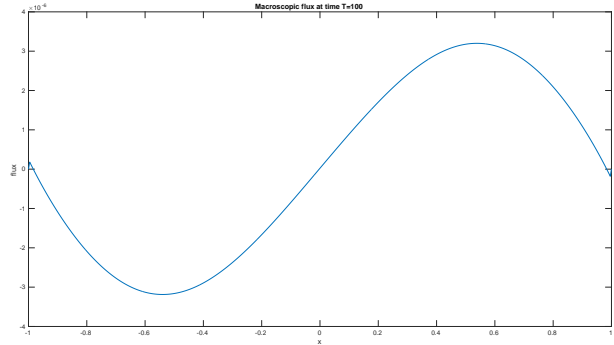}
 \caption{The macroscopic flux $J$ at time $T=100$.}
 \label{macroscopic_flux_chemotaxis}
\end{figure}

\subsection{Test case 2 : Radiative transport}
We consider the radiative transport equation \eqref{radiative_transport} on $x \in [0,1]$ and $v\in [-1,1]$.
The source term $q$ is given by 
\begin{equation}\label{eq:sourceq}
 q(x,v)=v\abs{v}+\sigma x\left(\abs{v}-\frac{1}{2}\right).
\end{equation}
When $\varepsilon=1$, we can have analytical nonzero steady state for this choice of $q(x,v)$. Though $q(x,v)$ in \eqref{eq:sourceq}
can take negative values, which is non-physical,
it can be used to test the scheme performance.
We use uniform grids in $x$ with $N_x=400$ and $32$ Gaussian quadrature for the velocity space.
The parameters are set to be
\begin{equation*}
 \sigma=1,a=0.01372,c=29.98,C_v=0.01.
\end{equation*}
For $I$, Dirichelet boundary condition on the left and specular reflection on the right are prescribed.
To guarantee the stability, the time step is given by 
\begin{equation*}
\Delta t=\left\{\begin{aligned}0.95\frac{\Delta x^2}{c} ,\quad&\mbox{for $\varepsilon<0.95\frac{\Delta x}{c}$},
\\0.95\varepsilon\frac{\Delta x}{c},\quad&\mbox{else}.\end{aligned}\right.
\end{equation*}
\subsubsection{AP property}
In this part, we show that the convergence order of the scheme is independent of $\varepsilon$. In Figure \ref{macro_rho_radiative_AP}, we plot 
$e_{\Delta x}:=\max_{x\in(0,1)}\|I_{\Delta x}(x,t)-I_{2\Delta x}(x,t)\|_1$ with respect to 
$\Delta x$ for different values of $\varepsilon$. Uniform convergence can be observed which confirms the AP property of the scheme.
\begin{figure}
 \centering
 \includegraphics[width=10cm]{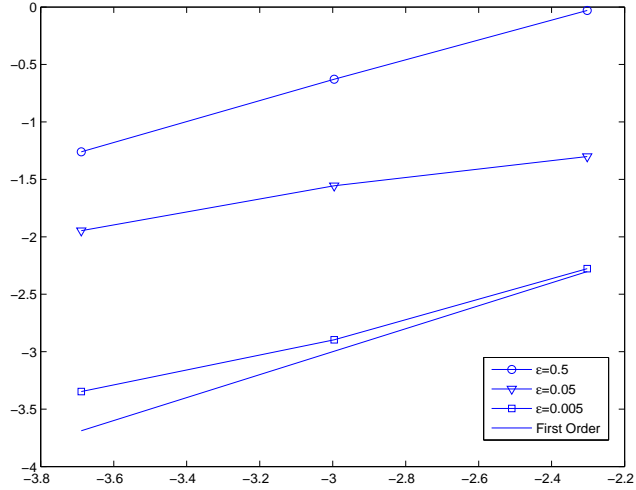}
 \caption{Convergence order with respect to $\Delta x$. Here the log-log plot of the mesh sizes with respect to the $L^1$ norm of the 
 probability density error in all directions is displayed. The mesh sizes are $1/10$, $1/20$, $1/40$, $1/80$ and 
 $e_{\Delta x}:=\max_{x\in(0,1)}\|I_{\Delta x}(x,t)-I_{2\Delta x}(x,t)\|_1$ are plotted for different $\varepsilon$.}
\end{figure}

\subsubsection{Model convergence with respect to $\varepsilon$}
 The theoretical analysis indicates that 
\begin{equation*}
I^{\varepsilon} \rightarrow \frac{\psi^0}{2} \mathbf{1}_{v \in V},\quad \varepsilon \rightarrow 0,
\end{equation*}
where $\psi^0=ac(T^0)^4$ and $T^0$ is the solution to the limiting equation \eqref{eq:limit_T}.
To verify the above convergence, the total densities $\rho^\varepsilon$ are displayed in Figure~\ref{macro_rho_radiative_AP}
for different values of $\varepsilon$ ranging from $10^{-1}$ 
to $10^{-6}$. As a comparison, we plot the solution of the limiting model in Figure~\ref{macro_rho_chemotaxis_AP}  and
can observe that $I^{\varepsilon}$ get closer to $\psi^0/2$ as $\varepsilon$ goes to zero.

The convergence order in $\varepsilon$ can be seen in Figure~\ref{log_error_radiative}, where we have plotted ${\|I^{\varepsilon} (x,v)-\rho^\varepsilon \mathbf{1}_{v \in V} \|}_2$, for different $\varepsilon$ in the logarithm scale. 
The numerical results indicate that the order of convergence is close to 0.6.

\begin{figure}[h]
\centering 
\includegraphics[width=10cm]{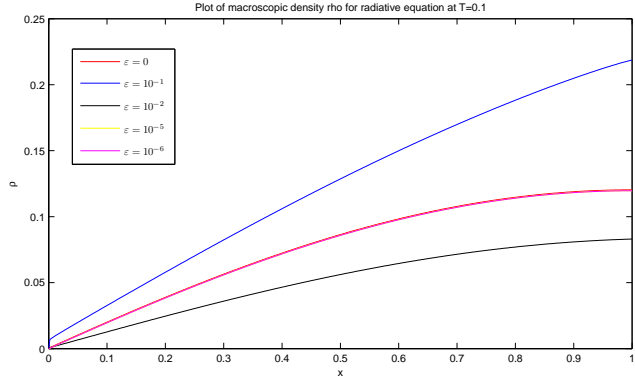}
\caption{Total density for different values of $\varepsilon$. $\varepsilon=0$ corresponds to the limiting nonlinear diffusion model \eqref{eq:limit_T}.}
\label{macro_rho_radiative_AP}
\end{figure}

\begin{figure}[h]
\centering 
\includegraphics[width=10cm]{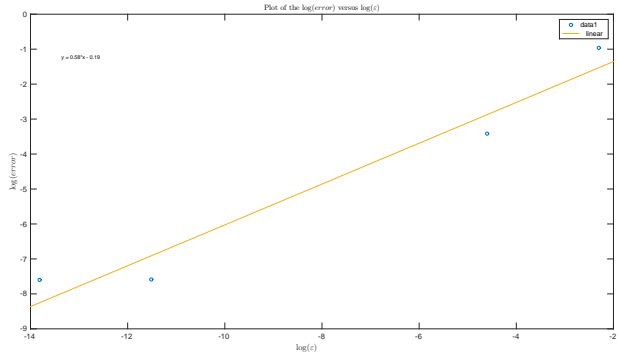}
\caption{Convergence order in $\varepsilon$ of the grey radiative transport model \eqref{radiative_transport}. }
\label{log_error_radiative}
\end{figure}
\subsubsection{WB property}
For the source term $q$ as in \eqref{eq:sourceq}, when $\varepsilon=1$, the steady state is given by 
\begin{equation*}
 I(x,v)=|v|x,\quad \rho=x.
\end{equation*}
The time evolution of the residues $r^n$ is given in Figure~\ref{residues_radiative}, it decrease and stabilize at $T=1$.
In Figure~\ref{rho_radiative}, the captured steady state is plotted.
\begin{figure}[h]
 \centering
    \includegraphics[width=10cm]{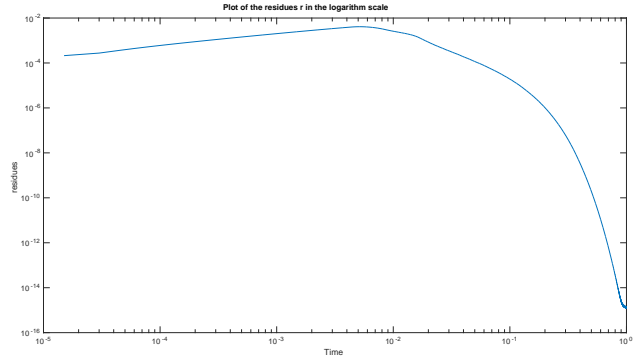}
    \caption{Evolution of residues with respect to time.}
    \label{residues_radiative}
\end{figure}
\begin{figure}[h]
 \centering
  \includegraphics[width=10cm]{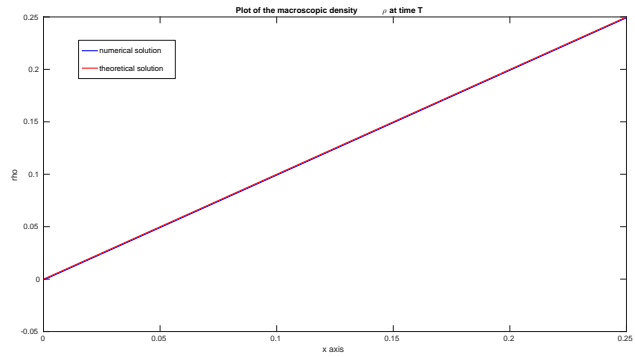}
 \caption{The total density of the steady state at time $T=1$ when $\varepsilon=1$.}
 \label{rho_radiative}
\end{figure}

\begin{section}{Conclusion}
A general framework of developing AP and WB schemes is proposed in this paper.
Two parameter regimes are considered:  the (advection) diffusion type (linear or nonlinear) limit when $\varepsilon<<1$ and
 the long time behavior of the hyperbolic equations with source terms when $\varepsilon=O(1)$.
The framework is composed of two steps. We first calculate a prediction with 
an AP scheme for the time evolutionary problem and then use the prediction to set up the inflow boundary conditions for the steady state problem 
in each cell. The steady state problem in each cell is solved by an AP scheme for the steady state equation, 
then the obtained outflows in each cell are used for modifying the numerical flux to achieve WB. 

The details of the discretizations
for transport equations for chemotaxis and gray radiative transfer are given. The numerical results in section 5
verify the AP and WB properties of our proposed scheme. In particular, the AP schemes for the time evolutionary transport equation for chemotaxis
and the corresponding steady state equation are new.

One may argue about the computational cost of our proposed framework, the advantage is that, any AP schemes for the time evolutionary problem
and steady state problem are applicable, thus we can choose those simple ones. First of all, the key principle underlying most WB 
schemes consists of using values constructed from a local discrete equilibrium, therefore if we are interested in unknown steady states,
it is hard to avoid the steady problem step. On the other hand, in the prediction step, we can use 
those AP scheme that can implement the implicit stiff terms explicitly as in \cite{Filbet}, or at least more efficient than using the Newton type solvers for
nonlinear algebraic systems.  \eqref{numerical_scheme} can be updated explicitly afterwards, the total computational cost is composed of
some explicit calculations plus
 solving the steady state problems.

The proposed framework can be extended to the hyperbolic system with relaxations, we will pursue more applications in the future.
\end{section}

\bibliographystyle{plain}


\begin{thebibliography}{}

\end{thebibliography}


\begin{thebibliography}{10}

\bibitem{Adams}
M.L. Adams.
\newblock Discontinuous finite element transport solutions in thick diffusive
  problems.
\newblock {\em Nuclear Science and Engineering}, 137 (2001), pp. 298--333.

\bibitem{Alt}
Wolgang Alt.
\newblock Biased random walk models for chemotaxis and related diffusion
  approximations.
\newblock {\em Journal of Mathematical Biology}, 9(2) (1980), pp. 147--177 .



\bibitem{BAMN}
Georgij Bispen, K.~R. Arun, Mária Luká?ová-Medvid?ová, and Sebastian
  Noelle.
\newblock Imex large time step finite volume methods for low froude number
  shallow water flows.
\newblock {\em Communications in Computational Physics}, 16(8) (2014), pp. 307--347.

\bibitem{BPV}
Ramaz Botchorishvili, Benoit Perthame, and Alexis Vasseur.
\newblock Equilibrium schemes for scalar conservation laws with stiff sources.
\newblock {\em Math. Comp.}, 72(241):(electronic), (2003), pp. 131--157.

\bibitem{Bokai}
Jos{\'e}~Antonio Carrillo and Bokai Yan.
\newblock An asymptotic preserving scheme for the diffusive limit of kinetic
  systems for chemotaxis.
\newblock {\em Multiscale Model. Simul.}, 11(1) (2013), pp. 336--361.

\bibitem{CMPS}
Fabio Chalub, Peter Markowich, Beno{\^{\i}}t Perthame, and Christian Schmeiser.
\newblock Kinetic models for chemotaxis and their drift-diffusion limits.
\newblock {\em Monatsh. Math.}, 142(1-2): (2004), pp. 123--141.

\bibitem{Chandrasekhar}
Subrahmanyan Chandrasekhar.
\newblock {\em Radiative transfer}.
\newblock Dover Publications, Inc., New York, 1960.

\bibitem{Ribot}
Cristiana Di~Russo, Roberto Natalini, and Magali Ribot.
\newblock Global existence of smooth solutions to a two-dimensional hyperbolic
  model of chemotaxis.
\newblock {\em Commun. Appl. Ind. Math.}, 1(1), (2010), pp. 92--109.

\bibitem{Almeida}
Casimir Emako, Luis Neves~de Almeida, and Nicolas Vauchelet.
\newblock Existence and diffusive limit of a two-species kinetic model of
  chemotaxis.
\newblock {\em Kinetic and Related Models}, 8(2), (2015), pp. 359--380.

\bibitem{Filbet}
Francis Filbet and Shi Jin.
\newblock A class of asymptotic-preserving schemes for kinetic equations and
  related problems with stiff sources.
\newblock {\em Journal of Computational Physics}, 229(20), (2010), pp. 7625--7648.

\bibitem{FilbetY}
Francis Filbet and Chang Yang.
\newblock Numerical simulations of kinetic models for chemotaxis.
\newblock {\em SIAM J. Sci. Comput.}, 36(3), (2014), pp. 348--366.

\bibitem{Gosse2}
Laurent Gosse.
\newblock Transient radiative transfer in the grey case: Well-balanced and
  asymptotic-preserving schemes built on case's elementary solutions.
\newblock {\em Journal of Quantitative Spectroscopy and Radiative Transfer},
  112(12):1995?2012, 2011.

\bibitem{Gosse1}
Laurent Gosse.
\newblock Asymptotic-preserving and well-balanced schemes for the 1{D}
  {C}attaneo model of chemotaxis movement in both hyperbolic and diffusive
  regimes.
\newblock {\em J. Math. Anal. Appl.}, 388(2) (2012), pp. 964--983.

\bibitem{Gosse}
Laurent Gosse.
\newblock A well-balanced scheme for kinetic models of chemotaxis derived from
  one-dimensional local forward-backward problems.
\newblock {\em Math. Biosci.}, 242(2) (2013), pp. 117--128.

\bibitem{Berg}
Berg Howard.
\newblock {\em E. coli in Motion}.
\newblock Biological and Medical Physics, Biomedical Engineering. Springer,
  2004.

\bibitem{HKS}
Hyung~Ju Hwang, Kyungkeun Kang, and Angela Stevens.
\newblock Drift-diffusion limits of kinetic models for chemotaxis: a
  generalization.
\newblock {\em Discrete Contin. Dyn. Syst. Ser. B}, 5(2), (2005), pp. 319--334.

\bibitem{SJin}
Shi Jin.
\newblock Asymptotic preserving (ap) schemes for multiscale kinetic and
  hyperbolic equations: a review. lecture notes for summer school on Ómethods
  and models of kinetic theory.
\newblock {\em Rivista di Mathematica della Universita di Parma, Porto Ercole
  (Grosseto, Italy)}, 2010.

\bibitem{Min}
Shi Jin, Min Tang, and Houde Han.
\newblock A uniformly second order numerical method for the one-dimensional
  discrete-ordinate transport equation and its diffusion limit with interface.
\newblock {\em Netw. Heterog. Media}, 4(1) (2009), pp. 35--65.

\bibitem{JW}
Shi Jin and Xin Wen.
\newblock An efficient method for computing hyperbolic systems with geometrical
  source terms having concentrations.
\newblock {\em J. Comput. Math.}, 22(2) (2004), pp. 230--249.
\newblock Special issue dedicated to the 70th birthday of Professor Zhong-Ci
  Shi.

\bibitem{Keller}
Evelyn~F. Keller and Lee~A. Segel.
\newblock Model for chemotaxis.
\newblock {\em Journal of Theoretical Biology}, 30(2) (1971), pp. 225 --234.

\bibitem{Kubelka}
Paul Kubelka.
\newblock New contributions to the optics of intensely light-scattering
  materials. {I}.
\newblock {\em J. Opt. Soc. Amer.}, 38 (1948), pp. 448--457.

\bibitem{Kurganov1}
Alexander Kurganov and Doron Levy.
\newblock Central-upwind schemes for the saint-venant system.
\newblock {\em ESAIM: Mathematical Modelling and Numerical Analysis},
  36(5) (2002), pp. 397--425.

\bibitem{LM}
E.~W. Larsen and J.~E. Morel.
\newblock Asymptotic solutions of numerical transport problems in optically
  thick,diffusive regimes ii.
\newblock {\em J. Comput. Phys.}, 69 (1989), pp. 212--236.

\bibitem{LMM}
Edward~W Larsen, J.E Morel, and Warren F~Miller Jr.
\newblock Asymptotic solutions of numerical transport problems in optically
  thick, diffusive regimes.
\newblock {\em Journal of Computational Physics}, 69(2) (1987), pp. 283 -- 324.

\bibitem{Leveque}
Randall~J. LeVeque.
\newblock Balancing source terms and flux gradients in high-resolution godunov
  methods: The quasi-steady wave-propagation algorithm.
\newblock {\em Journal of Computational Physics}, 146(1) (1998), pp. 346--365.

\bibitem{Mieussens}
Luc Mieussens.
\newblock On the asymptotic preserving property of the unified gas kinetic
  scheme for the diffusion limit of linear kinetic models.
\newblock {\em J. Comput. Phys.}, 253 (2013), pp. 138--156.

\bibitem{Natalini}
Roberto Natalini, Magali Ribot, and Monika Twarogowska.
\newblock A well-balanced numerical scheme for a one dimensional quasilinear
  hyperbolic model of chemotaxis.
\newblock {\em Commun. Math. Sci.}, 12(1) (2014), pp. 13--39.

\bibitem{OthmerAlt}
Hans Othmer, Stevens Dunbar, and Wolfgang Alt.
\newblock Models of dispersal in biological systems.
\newblock {\em J. Math. Biol.}, 26(3) (1988), pp. 263--298.

\bibitem{OthmerHill}
Hans Othmer and Thomas Hillen.
\newblock The diffusion limit of transport equations. {II}. {C}hemotaxis
  equations.
\newblock {\em SIAM J. Appl. Math.}, 62(4) (electronic), (2002), pp. 1222--1250.

\bibitem{PS}
B.~Perthame and C.~Simeoni.
\newblock A kinetic scheme for the saint-venant system with a source term.
\newblock {\em CALCOLO}, 38(4) (2001), pp. 201--231.

\bibitem{VCJS}
Jonathan Saragosti, Vincent Calvez, Nikolaos Bournaveas, Axel Buguin, Pascal
  Silberzan, and Beno{\^{\i}}t Perthame.
\newblock Mathematical description of bacterial traveling pulses.
\newblock {\em PLoS Comput. Biol.}, 6(8):e1000890, 12, (2010).

\bibitem{Saragosti}
Jonathan Saragosti, Vincent Calvez, Nikolaos Bournaveas, Beno{\^{\i}}t
  Perthame, Axel Buguin, and Pascal Silberzan.
\newblock Directional persistence of chemotactic bacteria in a traveling
  concentration wave.
\newblock {\em Proceedings of the National Academy of Sciences},
  108(39) (2011), pp. 16235--16240.

\bibitem{Sun}
Wenjun Sun, Song Jiang, and Kun Xu.
\newblock Asymptotic preserving unified gas kinetic scheme for grey radiative
  transfer equations.
\newblock {\em J Comput. Physics}, 285 (2015), pp. 265--279.



\bibitem{Xu}
Kun Xu.
\newblock A gas-kinetic \{BGK\} scheme for the Navier-Stokes equations and its
  connection with artificial dissipation and godunov method.
\newblock {\em Journal of Computational Physics}, 171(1) (2001), pp. 289--335.

\bibitem{Xu1}
Kun Xu.
\newblock A well-balanced gas-kinetic scheme for the shallow-water equations
  with source terms.
\newblock {\em Journal of Computational Physics}, 178(2) (2002), pp. 533 -- 562.

\bibitem{Xu2}
Kun Xu and Juan-Chen Huang.
\newblock A unified gas-kinetic scheme for continuum and rarefied flows.
\newblock {\em Journal of Computational Physics}, 229(20) (2010), pp. 7747--7764.

\end{thebibliography}
\end{document}